\documentclass[11pt]{amsart}
\usepackage{amsmath, amssymb}
\usepackage{amsfonts}
\usepackage{graphicx}
\usepackage{mathrsfs}
\usepackage[arrow,matrix,curve,cmtip,ps]{xy}
\usepackage{amsthm}
\usepackage{enumerate}
\usepackage{color}
\usepackage[colorlinks]{hyperref}
\usepackage{tikz}

\allowdisplaybreaks

\newtheorem{thm}{Theorem}[section]
\newtheorem{lem}[thm]{Lemma}
\newtheorem{prop}[thm]{Proposition}
\newtheorem{cor}[thm]{Corollary}
\newtheorem*{theorem*}{Theorem}
\theoremstyle{remark}
\newtheorem{rem}[thm]{Remark}

\newtheorem{defn}[thm]{Definition}
\newtheorem{ex}[thm]{Example}

\numberwithin{equation}{section}

\newcommand{\g}{\mathcal{G}}
\newcommand{\go}{\mathcal{G}^0}
\newcommand{\gl}{\mathcal{G}^1}

\newcommand{\pGo}{\Phi(G^0)}
\newcommand{\gh}{\mathcal{G}/(H,B)}
\newcommand{\pgo}{\Phi(\mathcal{G}^0)}
\newcommand{\pgl}{\Phi(\mathcal{G}^1)}
\newcommand{\sg}{\Phi_{\mathrm{sg}}(G^0)}

\begin{document}
\title[Primitive ideals and pure infiniteness]{Primitive ideals and pure infiniteness of ultragraph $C^*$-algebras}

\author[Hossein Larki]{Hossein Larki}

\address{Department of Mathematics\\
Faculty of Mathematical Sciences and Computer\\
Shahid Chamran University of Ahvaz\\
P.O. Box: 83151-61357\\
Ahvaz\\
 Iran}
\email{h.larki@scu.ac.ir}

%\thanks{This research was supported by NSF Grant DMS-1234567}

\date{\today}



\begin{abstract}
Let $\g$ be an ultragraph and let $C^*(\g)$ be the associated $C^*$-algebra introduced by Tomforde. For any gauge invariant ideal $I_{(H,B)}$ of $C^*(\g)$, we approach the quotient $C^*$-algebra $C^*(\g)/I_{(H,B)}$ by the $C^*$-algebra of finite graphs and prove versions of gauge invariant and Cuntz-Krieger uniqueness theorems for it. We then describe primitive gauge invariant ideals and determine purely infinite ultragraph $C^*$-algebras (in the sense of Kirchberg-R${\o}$rdam) via Fell bundles.
\end{abstract}

\maketitle

%%%%%%%%%%%%%%%%%%%%%%%%%%%%%%%%%%%%%%%%%%%%%%%%
\section{Introduction}

In order to bring graph $C^*$-algebras and Exel-Laca algebras together under one theory, Tomforde introduced in \cite{tom} the notion of ultragraphs and associated $C^*$-algebras. An ultragraph is basically a directed graph in which the range of each edge is allowed to be a nonempty set of vertices rather than a single vertex. However, the class of ultragraph $C^*$-algebras are strictly lager than the graph $C^*$-algebras as well as the Exel-Laca algebras (see \cite[Section 5]{tom2}). Due to some similarities, some of fundamental results for graph $C^*$-algebras, such as the Cuntz-Krieger and the gauge invariant uniqueness theorems, simplicity, and $K$-theory computation have been extended to the setting of ultragraphs \cite{tom,tom2}. In particular, by constructing a specific topological quiver $\mathcal{Q}(\g)$ from an ultragraph $\g$, Katsura et al. described some properties of the ultragraph $C^*$-algebra $C^*(\g)$ using those of topological quivers \cite{kat3}. They showed that every gauge invariant ideal of  $C^*(\g)$ is of the form $I_{(H,B)}$ corresponding to an admissible pair $(H,B)$ in $\g$.

Recall that for any gauge invariant ideal $I_{(H,B)}$ of a graph $C^*$-algebra $C^*(E)$, there is a (quotient) graph $E/(H,B)$ such that $C^*(E)/I_{(H,B)}\cong C^*(E/(H,B))$ (see \cite {bat2,bat}). So, the class of graph $C^*$-algebras contains such quotients, and results and properties of graph $C^*$-algebras may be applied for their quotients. For examples, some contexts such as simplicity, $K$-theory, primitivity, and topological stable rank are directly related to the structure of ideals and quotients.

Unlike the $C^*$-algebras of graphs and topological quivers, the class of ultragraph $C^*$-algebras $C^*(\g)$ is not closed under quotients. This causes some obstacles in studying the structure of ultragraph $C^*$-algebras. The initial aim of this article is to analyze the structure of the quotient $C^*$-algebras $C^*(\g)/I_{(H,B)}$ for any gauge invariant ideal $I_{(H,B)}$ of $C^*(\g)$. For the sake of convenience, we first introduce the notion of quotient ultragraph $\gh$ and a relative $C^*$-algebra $C^*(\gh)$ such that $C^*(\g)/I_{(H,B)}\cong C^*(\gh)$ and then prove the gauge invariant and the Cuntz-Krieger uniqueness theorems for $C^*(\gh)$. The uniqueness theorems help us to show when a representation of $C^*(\g)/I_{(H,B)}$ is injective. We see that the structure of $C^*(\gh)$ is close to that of graph $C^*$-algebras and we can use them to determine primitive gauge invariant ideals. Moreover, in Section 6, we consider the notion of pure infiniteness for ultragraph $C^*$-algebras in the sense of Kirchberg-R${\o}$rdam \cite{kirch00} which is directly related to the structure of quotients. We should note that the initial idea for definition of quotient ultragraphs has been inspired from \cite{jeo}.

The present article is organized as follows. We begin in Section 2 by giving some definitions and preliminaries about the ultragraphs and their $C^*$-algebras which will be used in the next sections. In Section 3, for any admissible pair $(H,B)$ in an ultragraph $\g$, we introduce the quotient ultragraph $\gh$ and an associated $C^*$-algebra $C^*(\gh)$. For this, the ultragraph $\g$ is modified by an extended ultragraph $\overline{\g}$ and we define an equivalent relation $\sim$ on $\overline{\g}$. Then $\gh$ is the ultragraph $\overline{\g}$ with the equivalent classes $\{[A]: A\in \overline{\g}^0\}$. In Section 4, by approaching with graph $C^*$-algebras, the gauge invariant and the Cuntz-Krieger uniqueness theorems will be proved for the quotient ultragraphs $C^*$-algebras. Moreover, we see that $C^*(\gh)$ is isometrically isomorphic to the quotient $C^*$-algebra $C^*(\g)/I_{(H,B)}$.

In Sections 5 and 6, using quotient ultragraphs, some graph $C^*$-algebra's techniques will be applied for the ultragraph $C^*$-algebras. In Section 5, we describe primitive gauge invariant ideals of $C^*(\g)$, whereas in Section 6, we characterize purely infinite ultragraph $C^*$-algebras (in the sense of \cite{kirch00}) via Fell bundles \cite{exe15,kwa17}.

%%%%%%%%%%%%%%%%%%%%%%%%%%%%%%%%%%%%%%%%%%%%%%%%

\section{preliminaries}

In this section, we review basic definitions and properties of ultragraph $C^*$-algebras which will be needed through the paper. For more details, we refer the reader to \cite{tom} and \cite{kat3}.

\begin{defn}[\cite{tom}]
An \emph{ultragraph} is a quadruple $\g=(G^0,\gl,r_\g,s_\g)$ consisting of a countable vertex set $G^0$, a countable edge set $\gl$, the source map $s_\g: \gl \rightarrow G^0$, and the range map $r_\g:\gl \rightarrow \mathcal{P}(G^0)\setminus \{\emptyset\}$, where $\mathcal{P}(G^0)$ is the collection of all subsets of $G^0$. If $r_\g(e)$ is a singleton vertex for each edge $e\in \gl$, then $\g$ is an ordinary (directed) graph.
\end{defn}

For our convenience, we use the notation $\go$ in the sense of \cite{kat3} rather than \cite{tom,tom2}. For any set $X$, a nonempty subcollection of the power set $\mathcal{P}(X)$ is said to be an \emph{algebra} if it is closed under the set operations $\cap$, $\cup$, and $\setminus$. If $\g$ is an ultragraph, the smallest algebra in $\mathcal{P}(G^0)$ containing $\{\{v\}:v\in G^0\}$ and $\{r_\g(e):e\in \gl\}$ is denoted by $\go$. We simply denote every singleton set $\{v\}$ by $v$. So, $G^0$ may be considered as a subset of $\go$.

\begin{defn}
For each $n\geq 1$, a \emph{path $\alpha$ of length $|\alpha|=n$} in $\g$ is a sequence $\alpha=e_1\ldots e_n$ of edges such that $s(e_{i+1})\in r(e_i)$ for $1\leq i\leq n-1$. If also $s(e_1)\in r(e_n)$, $\alpha$ is called a \emph{loop} or a \emph{closed path}. We write $\alpha^0$ for the set $\{s_\g(e_i): 1\leq i\leq n\}$. The elements of $\go$ are considered as the paths of length zero. The set of all paths in $\g$ is denoted by $\g^*$. We may naturally extend the maps $s_\g,r_\g$ on $\g^*$ by defining $s_\g(A)=r_\g(A)=A$ for $A\in \go$, and $r_\g(\alpha)=r_\g(e_n)$, $s_\g(\alpha)=s_\g(e_1)$ for each path $\alpha=e_1\ldots e_n$.
\end{defn}

\begin{defn}[\cite{tom}]\label{defn2.3}
Let $\g$ be an ultragraph. A \emph{Cuntz-Krieger $\g$-family} is a set of partial isometries $\{s_e:e\in\gl\}$ with mutually orthogonal ranges and a set of projections $\{p_A:A\in\go\}$ satisfying the following relations:
\begin{enumerate}
\item[(UA1)] $p_\emptyset=0$, $p_A p_B=p_{A\cap B}$, and $p_{A\cup B}=p_A +p_B-p_{A\cap B}$ for all $A,B\in\go$,
\item[(UA2)] $s_e^* s_e=p_{r_\g(e)}$ for $e\in\gl$,
\item[(UA3)] $s_e s_e^*\leq p_{s_\g(e)}$ for $e\in \gl$, and
\item[(UA4)] $p_v=\sum_{s_\g(e)=v}s_e s_e^*$ whenever $0<|s_\g^{-1}(v)|<\infty$.
\end{enumerate}
The $C^*$-algebra $C^*(\g)$ of $\g$ is the (unique) $C^*$-algebra generated by a universal Cuntz-Krieger $\g$-family.
\end{defn}

By \cite[Remark 2.13]{tom}, we have
$$C^*(\g)=\overline{\mathrm{span}}\left\{s_\alpha p_A s_\beta^*:\alpha,\beta\in \g^*, A\in \go, \mathrm{~and~} r_\g(\alpha)\cap r_\g(\beta)\cap A\neq \emptyset \right\},$$
where $s_\alpha:=s_{e_1}\ldots s_{e_n}$ if $\alpha=e_1\ldots e_n$, and $s_\alpha:=p_A$ if $\alpha=A$.

\begin{rem}\label{rem2.4}
As noted in \cite[Section 3]{tom}, every graph $C^*$-algebra is an ultragraph $C^*$-algebra. Recall that if $E=(E^0,E^1,r_E,s_E)$ is a directed graph, a collection $\{s_e,p_v:v\in E^0, e\in E^1\}$ containing mutually orthogonal projections $p_v$ and partial isometries $s_e$ is called a \emph{Cuntz-Krieger $E$-family} if
\begin{enumerate}
\item[(GA1)] $s_e^* s_e=p_{r_E(e)}$ for all $e\in E^1$,
\item[(GA2)] $s_e s_e^*\leq p_{s_E(e)}$ for all $e\in E^1$, and
\item[(GA3)] $p_v=\sum_{s_E(e)=v}s_e s_e^*$ for every vertex $v\in E^0$ with $0<|s_E^{-1}(v)|<\infty$.
\end{enumerate}
We denote by $C^*(E)$ the universal $C^*$-algebra generated by a Cuntz-Krieger $E$-family.
\end{rem}

By the universal property, $C^*(\g)$ admits the \emph{gauge action} of the unit circle $\mathbb{T}$. By an \emph{ideal}, we mean a closed two-sided ideal. Using the properties of quiver $C^*$-algebras \cite{kat3}, the gauge invariant ideals of $C^*(\g)$ were characterized in \cite[Theorem 6.12]{kat3} via a one-to-one correspondence with the admissible pairs of $\g$ as follows.

\begin{defn}\label{defn2.5}
A subset $H\subseteq \go$ is said to be \emph{hereditary} if the following properties holds:
\begin{enumerate}
\item[(H1)] $s_\g(e)\in H$ implies $r_\g(e)\in H$ for all $e\in \gl$.
\item[(H2)] $A\cup B\in H$ for all $A,B\in H$.
\item[(H3)] If $A\in H,~B\in\go$, and $B\subseteq A$, then $B\in H$.
\end{enumerate}
Moreover, a subset $H\subseteq \go$ is called \emph{saturated} if for any $v\in G^0$ with $0<|s_\g^{-1}(v)|<\infty$, then $r_\g(s_\g^{-1}(v))\subseteq H$ implies $v\in H$. The \emph{saturated hereditary closure} of a subset $H\subseteq \go$ is the smallest hereditary and saturated subset $\overline{H}$ of $\go$ containing $H$.
\end{defn}

Let $H$ be a saturated hereditary subset of $\go$. The set of \emph{breaking vertices of $H$} is denoted by
$$B_H:=\left\{ w\in G^0: |s_\g^{-1}(w)|=\infty ~~\mathrm{but}~~ 0<|r_\g(s_\g^{-1}(w))\cap (\go\setminus H)|<\infty\right\}.$$
An \emph{admissible pair $(H,B)$ in $\g$} is a saturated hereditary set $H\subseteq \go$ together with a subset $B\subseteq B_H$. For any admissible pair $(H,B)$ in $\g$, we define the ideal $I_{(H,B)}$ of $C^*(\g)$ generated by
$$\{p_A:A\in \go\} \cup \left\{p_w^H:w\in B\right\},$$
where $p_w^H:=p_w-\sum_{s_\g(e)=w,~ r_\g(e)\notin H}s_e s_e^*$. Note that the ideal $I_{(H,B)}$ is gauge invariant and \cite[Theoerm 6.12]{kat3} implies that every gauge invariant ideal $I$ of $C^*(\g)$ is of the form $I_{(H,B)}$ by setting
$$H:=\left\{A:p_A\in I\right\} ~~ \mathrm{and} ~~ B:=\left\{ w\in B_H: p_w^H\in I\right\}.$$

%%%%%%%%%%%%%%%%%%%%%%%%%%%%%%%%%%%%%%%%%%%%%%%%

\section{Quotient Ultragraphs and their $C^*$-algebras}

In this section, for any admissible pair $(H,B)$ in an ultragraph $\g$, we introduce the quotient ultragraph $\g/(H,B)$ and its relative $C^*$-algebra $C^*(\g/(H,B))$. We will show in Proposition \ref{prop5.1} that $C^*(\g/(H,B))$ is isomorphic to the quotient $C^*$-algebra $C^*(\g)/ I_{(H,B)}$.

Let us fix an ultragraph $\g=(G^0,\g^0,r_\g,s_\g)$ and an admissible pair $(H,B)$ in $\g$. For defining our quotient ultragraph $\gh$, we first modify $\g$ by an extended ultragraph $\overline{\g}$ such that their $C^*$-algebras coincide. For this, add the vertices $\{w':w\in B_H\setminus B\}$ to $G^0$ and denote $\overline{A}:=A\cup \{w':w\in A\cap (B_H\setminus B)\}$ for each $A\in \go$. We now define the new ultragraph $\overline{\g}=(\overline{G^0},\overline{\g}^1,\overline{r}_\g,\overline{s}_\g)$ by
\begin{align*}
\overline{G}^0&:= G^0\cup \{w':w\in B_H\setminus B\},\\
\overline{\g}^1&:=\gl,
\end{align*}
the source map
\[\overline{s}_\g(e):=\left\{
          \begin{array}{ll}
            (s_\g(e))' & \mathrm{if}~~ s_\g(e)\in B_H\setminus B ~~\mathrm{and}~~ r_\g(e)\in H \\
            s_\g(e) & \mathrm{otherwise},
          \end{array}
        \right.
\]
and the rang map $\overline{r}_\g(e):=\overline{r_\g(e)}$ for every $e\in \g^1$. In Proposition \ref{prop3.2} below, we will see that the $C^*$-algebras of $\g$ and $\overline{\g}$ coincide.

\begin{ex}
Suppose $\g$ is the ultragraph
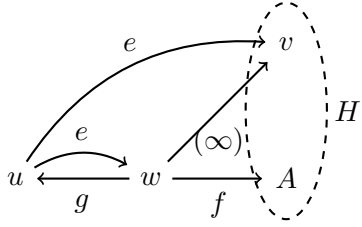
\begin{center}
\begin{tikzpicture}[scale=1.8,shorten >=1pt,auto,thick]
\node (w) at (0,0) {$w$};
\node (u) at (-1,0) {$u$};
\node (v) at (1,1) {$v$};
\node (A) at (1,0) {$A$};

\draw[->] (w) -- node[below=1pt] {$f$} (A);
\draw[->] (w) -- node[below=2pt] {$(\infty)$} (v);
\path[->] (u) edge[bend left] node[above=2pt] {$e$} (v);
\path[->] (u) edge[bend left] node[above=1pt] {$e$} (w);
\path[->] (w) edge node[below=2pt] {$g$} (u);

\draw[dashed] (1,.5) ellipse (3mm and 8mm)
node[right=5mm] {$H$};

\end{tikzpicture}
\end{center}
where $(\infty)$ indicates infinitely many edges. If $H$ is the saturated hereditary subset of $\go$ containing $\{v\}$ and $A$, then we have $B_H=\{w\}$. For $B:=\emptyset$, consider the admissible pair $(H,\emptyset)$ in $\g$. Then the ultragraph $\overline{\g}$ associated to $(H,\emptyset)$ would be
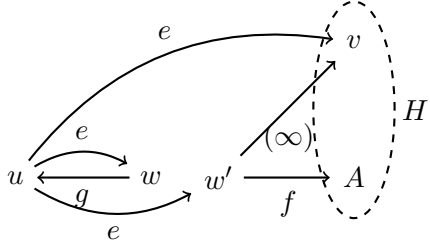
\begin{center}
\begin{tikzpicture}[scale=1.8,shorten >=1pt,auto,thick]
\node (wG) at (3.5,0) {$w$};
\node (w'G) at (4,0) {$w'$};
\node (uG) at (2.5,0) {$u$};
\node (vG) at (5,1) {$v$};
\node (AG) at (5,0) {$A$};

\draw[->] (w'G) -- node[below=2pt] {$(\infty)$} (vG);
\draw[->] (w'G) -- node[below=1pt] {$f$} (AG);
\path[->] (uG) edge[bend left] node[above=2pt] {$e$} (vG);
\path[->] (uG) edge[bend left] node[above=1pt] {$e$} (wG);
\path[->] (wG) edge node[below=.5pt] {$g$} (uG);
\path[->] (uG) edge[bend right] node[below=1pt] {$e$} (w'G);

\draw[dashed] (5,.5) ellipse (3mm and 8mm)
node[right=5mm] {$H$};
\end{tikzpicture}
\end{center}
Indeed, since $B_H\setminus B=\{w\}$, for constructing $\overline{\g}$ we first add a vertex $w'$ to $\g$. We then define
\begin{align*}
\overline{r}_\g(f)&:=\overline{A}=A,\\
\overline{r}_\g(e)&:=\overline{\{v,w\}}=\{v,w,w'\}, \mathrm{~ and}\\
\overline{r}_\g(g)&:=\overline{\{u\}}=\{u\}.
\end{align*}
For the source map $\overline{s}_\g$, for example, since $s_\g(f)\in B_H\setminus B$ and $r_\g(f)\in H$, we may define $\overline{s}_\g(f):=w'$. Note that the range of each edge emitted by $w'$ belongs to $H$.
\end{ex}

As usual, we write $\overline{\g}^0$ for the algebra generated by the elements of $\overline{G}^0 \cup \{\overline{r}_\g(e):e\in \overline{\g}^1\}$. Note that $\overline{A}=A$ for every $A\in H$, and hence, $H$ would be a saturated hereditary subset of $\overline{\g}^0$ as well. Moreover, the set of breaking vertices of $H$ in $\overline \g$ coincides with $B$ (meaning $B_H^{\overline{\g}}=B$).

\begin{rem}\label{rem3.2}
Suppose that $C^*(\g)$ is generated by a Cuntz-Krieger $\g$-family $\{s_e,p_A: A\in \go,e\in\gl\}$. If a family $M=\{S_e,P_v,P_{\overline{A}}: v\in G^0,A\in \g^0, e\in \overline{\g}^1\}$ in a $C^*$-algebra $X$ satisfies relations (UA1)-(UA4) in Definition \ref{defn2.3}, we may generate a Cuntz-Krieger $\overline{\g}$-family $N=\{S_e,P_A:A\in \overline{\g}^0,e\in \overline{\g}^1\}$ in $X$. For this, since $\overline{\g}^0$ is the algebra generated by $\{v,w',\overline{r}_\g(e): v\in G^0,w\in B_H\setminus B, e\in \overline{\g}^1\}$, it suffices to define
\begin{align*}
P_{A\cap B}&:=P_A P_B\\
P_{A\cup B}&:=P_A+P_B-P_A P_B\\
P_{A\setminus B}&:=P_A-P_A P_B
\end{align*}
and generate projections $P_A$ for all $A\in \overline{\g}^0$. Then $N$ is a Cuntz-Krieger $\overline{\g}$-family in $X$, and the $C^*$-subalgebras generated by $M$ and $N$ coincide.
\end{rem}

\begin{prop}\label{prop3.2}
Let $\g$ be an ultragraph, and let $(H,B)$ be an admissible pair in $\g$. If $\overline{\g}$ is the extended ultragraph as above, then $C^*(\g)\cong C^*(\overline{\g})$.
\end{prop}

\begin{proof}
Suppose that $C^*(\g)=C^*(t_e,q_A)$ and $C^*(\overline{\g})=C^* (s_e,p_C)$. If we define
$$\begin{array}{ll}
  P_v:=q_v & \mathrm{for}~~ v\in G^0\setminus (B_H\setminus B) \\
  P_w:=\sum_{\substack{s_\g(e)=w \\ r_\g(e)\notin H}}t_et_e^* &  \mathrm{for}~~ w\in B_H\setminus B \\
  P_{w'}:=q_w-\sum_{\substack{s_\g(e)=w \\ r_\g(e)\notin H}}t_et_e^* & \mathrm{for}~~ w\in B_H\setminus B \\
  P_{\overline{A}}:=q_A & \mathrm{for}~~ \overline{A}\in \overline{\g}^0 \\
  S_e:=t_e & \mathrm{for}~~ e\in \overline{\g}^1,
\end{array}
$$
then, by Remark \ref{rem3.2}, the family
$$\left\{P_v,P_w,P_{w'},P_{\overline{A}},S_e: v\in G^0\setminus (B_H\setminus B),~w\in B_H\setminus B,~ \overline{A}\in \overline{\g}^0,~ e\in \overline{\g}^1\right\}$$
induces a Cuntz-Krieger $\overline{\g}$-family in $C^*(\g)$. Since all vertex projections of this family are nonzero (which follows all set projections $P_A$ are nonzero for $\emptyset\neq A\in\overline{\g}^0$), the gauge-invariant uniqueness theorem \cite[Theorem 6.8]{tom} implies that the $*$-homomorphism $\phi:C^*(\overline{\g})\rightarrow C^*(\g)$ with $\phi(p_*)=P_*$ and $\phi(s_*)=S_*$ is injective. On the other hand, the family generates $C^*(\g)$, and hence, $\phi$ is an isomorphism.
\end{proof}

To define a quotient ultragraph $\gh$, we use the following equivalent relation on $\overline{\g}$.

\begin{defn}\label{defn3.3}
Suppose that $(H,B)$ is an admissible pair in $\g$, and that $\overline{\g}$ is the extended ultragraph as above. We define the relation $\sim$ on $\overline{\g}^0$ by
$$A\sim B ~~ \Longleftrightarrow ~~ \exists V\in H ~\mathrm{such ~ that} ~ A\cup V=B\cup V.$$
Note that $A\sim B$ if and only if both sets $A\setminus B$ and $B\setminus A$ belong to $H$.
\end{defn}

The following is an analogous version of \cite[Proposition 3.6]{jeo}.

\begin{lem}\label{lem3.4}
The relation $\sim$ is an equivalent relation on $\overline{\g}^0$. Furthermore, the operations
$$[A]\cup[B]:=[A\cup B], ~ [A]\cap [B]:=[A\cap B], ~ \mathrm{and}~ [A]\setminus [B]:= [A\setminus B]$$
are well-defined on the equivalent classes $\{[A]: A\in \overline{\g}^0\}$.
\end{lem}

\begin{defn}
Let $\g$ be an ultragraph, let $(H,B)$ be an admissible pair in $\g$, and consider the equivalent relation of Definition \ref{defn3.3} on the extended ultragraph $\overline{\g}=(\overline{G}^0,\overline{\g}^1,\overline{r}_\g,\overline{s}_\g)$. The \emph{quotient ultragraph of $\g$ by $(H,B)$} is the quintuple $\gh=(\pGo, \pgo, \pgl,r,s)$, where
\begin{align*}
\Phi(G^0)&:=\left\{[v]:v\in G^0\setminus H\right\}\cup \left\{[w']:w\in B_H\setminus B\right\},\\
\Phi(\go)&:=\left\{[A]:A\in \overline{\g}^0 \right\},\\
\Phi(\gl)&:=\left\{e\in \overline{\g}^1:\overline{r}_\g(e)\notin H \right\},
\end{align*}
and $r:\pgl \rightarrow \pgo$, $s:\pgl \rightarrow \pGo$ are the range and source maps defined by
$$r(e)=[\overline{r}_\g(e)] \hspace{5mm} \mathrm{and} \hspace{5mm} s(e):=[\overline{s}_\g(e)].$$
We refer to $\pGo$ as the vertices of $\gh$.
\end{defn}

\begin{rem}
Lemma \ref{lem3.4} implies that $\pgo$ is the smallest algebra containing
$$\left\{[v],[w']:v\in G^0\setminus H, w\in B_H\setminus B\right\}\cup \left\{[\overline{r}_\g(e)]: e\in \overline{\g}^1\right\}.$$
\end{rem}
{\bf Notation.}\begin{enumerate}
\item For every vertex $v\in \overline{\g}^0\setminus H$, we usually denote $[v]$ instead of $[\{v\}]$.
\item For $A,B\in\overline{\g}^0$, we write $[A]\subseteq [B]$ whenever $[A]\cap [B]=[A]$.
\item Through the paper, we will denote the range and the source maps of $\g$ by $r_\g ,s_\g$, those of $\overline{\g}$ by $\overline{r}_\g,\overline{s}_\g$, and those of $\gh$ by $r,s$.
\end{enumerate}

Now we introduce representations of quotient ultragraphs and their relative $C^*$-algebras.

\begin{defn}\label{defn3.8}
Let $\gh$ be a quotient ultragraph. A \emph{representation of $\gh$} is a set of partial isometries $\{T_e:e\in \pgl\}$ and a set of projections $\{Q_{[A]}: [A]\in \pgo\}$ which satisfy the following relations:
\begin{itemize}
\item[(QA1)]{$Q_{[\emptyset]}=0$, $Q_{[A\cap B]}=Q_{[A]} Q_{[B]}$, and $Q_{[A\cup B]}=Q_{[A]} +Q_{[B]}-Q_{[A\cap B]}$.}
\item[(QA2)]{$T_e^* T_e=Q_{r(e)}$ and $T_e^* T_f=0$ when $e\neq f$.}
\item[(QA3)]{$T_e T_e^*\leq Q_{s(e)}$.}
\item[(QA4)]{$Q_{[v]}=\sum_{s(e)=[v]}T_e T_e^*$, whenever $0<|s^{-1}([v])|<\infty$.}
\end{itemize}
We denote by $C^*(\g/(H,B))$ the universal $C^*$-algebra generated by a representation $\{t_e,q_{[A]}:[A]\in \pgo,e\in \pgl\}$ which exists by Theorem \ref{thm3.11} below.
\end{defn}

Note that if $\alpha=e_1\ldots e_n$ is a path in $\overline{\g}$ and $\overline{r}_\g(\alpha)\notin H$, then the hereditary property of $H$ yields $\overline{r}_\g(e_i)\notin H$, and so $e_i\in \pgl$ for all $1\leq i\leq n$. In this case, we denote $t_\alpha:=t_{e_1}\ldots t_{e_n}$. Moreover, we define
$$(\gh)^*:=\left\{[A]: [A]\neq [\emptyset]\right\}\cup \left\{\alpha\in\overline{\g}^*: r(\alpha)\neq [\emptyset]\right\}$$
as the set of finite paths in $\gh$ and we can extend the maps $s,r$ on $(\gh)^*$ by setting
$$s([A]):=r([A]):=[A] ~~~ \mathrm{and} ~~~ s(\alpha):=s(e_1),~ r(\alpha):=r(e_n). $$

The proof of next lemma is similar to the arguments of \cite[Lemmas 2.8 and 2.9]{tom}.

\begin{lem}\label{lem3.10}
Let $\gh$ be a quotient ultragraph and let $\{T_e,Q_{[A]}\}$ be a representation of $\gh$. Then any nonzero word in $T_e$, $Q_{[A]}$, and $T_f^*$ may be written as a finite linear combination of the forms $T_\alpha Q_{[A]} T_\beta^*$ for $\alpha,\beta\in (\gh)^*$ and $[A]\in \pgo$ with $[A]\cap r(\alpha)\cap r(\beta)\neq [\emptyset]$.
\end{lem}

\begin{thm}\label{thm3.11}
Let $\gh$ be a quotient ultragraph. Then there exists a (unique up to isomorphism) $C^*$-algebra $C^*(\gh)$ generated by a universal representation $\{t_e,q_{[A]}: [A]\in \pgo,e\in \pgl\}$ for $\gh$. Furthermore, all the $t_e$'s and $q_{[A]}$'s are nonzero for $[\emptyset]\neq [A]\in \pgo$ and $ e\in \pgl$.
\end{thm}

\begin{proof}
By a standard argument similar to the proof of \cite[Theorem 2.11]{tom}, we may construct such universal $C^*$-algebra $C^*(\gh)$. Note that the universality implies that $C^*(\gh)$ is unique up to isomorphism. To show the last statement, we generate an appropriate representation for $\gh$ as follows. Suppose $C^*(\overline{\g})=C^*(s_e,p_A)$ and consider $I_{(H,B)}$ as an ideal of $C^*(\overline{\g})$ by the isomorphism in Proposition \ref{prop3.2}. If we define
$$\left\{
    \begin{array}{ll}
      Q_{[A]}:=p_A+I_{(H,B)} & \mathrm{for~~}[A]\in \pgo \\
      T_{e}:=s_e+I_{(H,B)} & \mathrm{for}~~ e\in \pgl,
    \end{array}
  \right.
$$
then the family $\{T_{e},Q_{[A]}:[A]\in \pgo,e\in \pgl\}$ is a representation for $\gh$ in the quotient $C^*$-algebra $C^*(\overline{\g})/I_{(H,B)}$. Note that the definition of $Q_{[A]}$'s is well-defined. Indeed, if $A_1\cup V=A_2 \cup V$ for some $V\in H$, then $p_{A_1}+p_{V\setminus A_1}=p_{A_2}+p_{V\setminus A_2}$ and hence $p_{A_1}+I_{(H,B)}=p_{A_2}+I_{(H,B)}$ by the facts $V\setminus A_1, V\setminus A_2\in H$.

Moreover, all elements $Q_{[A]}$ and $T_e$ are nonzero for $[\emptyset]\ne[A]\in\pgo$, $e\in\pgl$. In fact, if $Q_{[A]}=0$, then $p_A\in I_{(H,B)}$ and we get $A\in H$ by \cite[Theorem 6.12]{kat3}. Also, since $T_e^* T_e=Q_{r(e)}\ne0$, all partial isometries $T_e$ are nonzero.

Now suppose that $C^*(\gh)$ is generated by the family $\{t_e,q_{[A]}:[A]\in \pgo, ~ e\in \pgl\}$. By the universality of $C^*(\gh)$, there is a $*$-homomorphism $\phi:C^*(\gh)\rightarrow C^*(\overline{\g})/I_{(H,B)}$ such that $\phi(t_e)=T_e$ and $\phi(q_{[A]})=Q_{[A]}$, and thus, all elements $\{t_e,q_{[A]}:[\emptyset]\ne[A]\in \pgo, ~ e\in \pgl\}$ are nonzero.
\end{proof}

Note that, by a routine argument, one may obtain
$$C^*(\gh)=\overline{\mathrm{span}}\big\{t_\alpha q_{[A]}t_\beta^*: \alpha,\beta\in (\gh)^*,~ r(\alpha)\cap [A]\cap r(\beta)\neq [\emptyset]\big\}.$$

%%%%%%%%%%%%%%%%%%%%%%%%%%%%%%%%%%%%%%%%%%%%%%%%

\section{Uniqueness Theorems}

After defining the $C^*$-algebra of quotient ultragraphs, in this section, we prove the gauge invariant and the Cuntz-Krieger uniqueness theorems for them. To do this, we approach to a quotient ultragraph $C^*$-algebra by graph $C^*$-algebras and then apply the corresponding uniqueness theorems for graph $C^*$-algebras. This approach is a developed version of the dual graph method of \cite[Section 2]{rae} and \cite[Section 5]{tom} with more complications. In particular, we show that the $C^*$-algebra $C^*(\gh)$ is isomorphic to the quotient $C^*(\g)/I_{(H,B)}$, and the uniqueness theorems may applied for such quotients.

We fix again an ultragraph $\g$, an admissible pair $(H,B)$ in $\g$, and the quotient ultragraph $\gh=(\pGo,\pgo,\pgl,r,s)$.

\begin{defn}
We say that a vertex $[v]\in \pGo$ is a \emph{sink} if $s^{-1}([v])=\emptyset$. If $[v]$ only emits finitely many edges of $\pgl$, $[v]$ is called a \emph{regular vertex}. Any non-regular vertex is called a \emph{singular vertex}. The set of singular vertices in $\pGo$ is denoted by $$\sg:=\big\{[v]\in \pGo: |s^{-1}([v])|=0~ \mathrm{or}~\infty\big\}.$$
\end{defn}

Let $F$ be a finite subset of $\sg \cup \pgl$. Write $F^0:=F\cap \sg$ and $F^1:=F\cap \pgl=\{e_1,\ldots,e_n\}$. We want to construct a special graph $G_F$ such that $C^*(G_F)$ is isomorphic to $C^*(t_e,q_{[v]}:[v]\in F^0,e\in F^1)$. For each $\omega=(\omega_1,\ldots, \omega_n)\in \{0,1\}^n\setminus \{0^n\}$, we write
$$r(\omega):=\bigcap_{\omega_i=1}r(e_i)\setminus \bigcup_{\omega_j=0}r(e_j) \mathrm{~~and~~} R(\omega):=r(\omega)\setminus \bigcup_{[v]\in F^0}[v].$$
Note that $r(\omega)\cap r(\nu)=[\emptyset]$ for distinct $\omega,\nu\in \{0,1\}\setminus \{0^n\}$. If
\begin{multline*}
    \Gamma_0:=\big\{\omega\in \{0,1\}^n\setminus\{0^n\}: \exists [v_1],\ldots,[v_m]\in\pgo $ such  that $\\
 R(\omega)=\bigcup_{i=1}^m[v_i] $ and $ \emptyset\neq s^{-1}([v_i])\subseteq F^1 $ for $ 1\leq i\leq m\big\},
\end{multline*}
we consider the finite set
\[\Gamma:=\left\{\omega\in \{0,1\}^n\setminus \{0^n\}: R(\omega)\neq [\emptyset] \mathrm{~and ~} \omega\notin \Gamma_0 \right\}.\]

Now we define the finite graph $G_F=(G_F^0,G_F^1,r_F,s_F)$ containing the vertices $G_F^0:=F^0 \cup F^1 \cup \Gamma$ and the edges
\begin{align*}
G_F^1:=&\left\{(e,f)\in F^1\times F^1: s(f)\subseteq r(e) \right\}\\
       &\cup \left\{(e,[v])\in F^1\times F^0: [v]\subseteq r(e) \right\}\\
       &\cup \left\{(e,\omega)\in F^1\times \Gamma: \omega_i=1 \mathrm{~~when~~} e=e_i \right\}
\end{align*}
with the source map $s_F(e,f)=s_F(e,[v])=s_F(e,\omega)=e$, and the range map $r_F(e,f)=f$, $r_F(e,[v])=[v]$, $r_F(e,\omega)=\omega$.

\begin{prop}\label{prop4.2}
Let $\gh$ be a quotient ultragraph and let $F$ be a finite subset of $\sg \cup \pgl$. If $C^*(\gh)=C^*(t_e,q_{[A]})$, then the elements
$$\begin{array}{ccc}
    Q_e:=t_et_e^*, & Q_{[v]}:=q_{[v]}(1-\sum_{e\in F^1}t_et_e^*), & Q_\omega:=q_{R(\omega)}(1-\sum_{e\in F^1}t_et_e^*) \\
    T_{(e,f)}:=t_eQ_f, & T_{(e,[v])}:=t_e Q_{[v]}, & T_{(e,\omega)}:=t_e Q_\omega
  \end{array}
$$
form a Cuntz-Krieger $G_F$-family generating the $C^*$-subalgebra $C^*(t_e,q_{[v]}:[v]\in F^0,e\in F^1)$ of $C^*(\gh)$. Moreover, all projections $Q_*$ are nonzero.
\end{prop}

\begin{proof}
We first note that all the projections $Q_e$, $Q_{[v]}$, and $Q_\omega$ are nonzero. Indeed, each $[v]\in F^0$ is a singular vertex in $\gh$, so $Q_{[v]}$ is nonzero. Also, by definition, for every $\omega\in \Gamma$ we have $\omega\notin \Gamma_0$ and $R(\omega)\neq [\emptyset]$. Hence, for any $\omega\in \Gamma$, if there is an edge $f\in \pgl\setminus F^1$ with $s(f)\subseteq R(\omega)$, then $0\neq t_ft_f^*\leq Q_{\omega}$. If there is a sink $[w]$ such that $[w]\subseteq R(\omega)=r(\omega)\setminus \bigcup F^0$, then $0\neq q_{[w]}\leq q_{R(\omega)}(1-\sum_{e\in F^1}t_et_e^*)=Q_{\omega}$. Thus $Q_{\omega}$ is nonzero in either case. In addition, the projections $Q_e$, $Q_{[v]}$, and $Q_\omega$ are mutually orthogonal because of the factor $1-\sum_{e\in F^1}t_et_e^*$ and the definition of $R(\omega)$.

Now we show the collection $\{T_x,Q_a:a\in G_F^0, x\in G_F^1\}$ is a Cuntz-Krieger $G_F$-family by checking the relations (GA1)-(GA3) in Remark \ref{rem2.4}.

\underline{(GA1)}: Since $Q_{[v]},Q_\omega \leq q_{r(e)}$ for $(e,[v]),(e,\omega)\in G_F^1$, we have
$$T^*_{(e,f)}T_{(e,f)}=Q_f t_e^* t_e Q_f=t_f t_f^* q_{r(e)} t_f t_f^*= t_f q_{r(f)} t_f^*=Q_f,$$
$$T^*_{(e,[v])}T_{(e,[v])}=Q_{[v]}t_e^* t_e Q_{[v]}=Q_{[v]} q_{r(e)} Q_{[v]}=Q_{[v]},$$
and
$$T^*_{(e,\omega)}T_{(e,\omega)}=Q_\omega t_e^* t_e Q_\omega=Q_\omega q_{r(e)} Q_\omega= Q_\omega.$$

\underline{(GA2)}: This relation may be checked similarly.

\underline{(GA3)}: Note that any element of $F^0 \cup \Gamma$ is a sink in $G_F$. So, fix some $e_i\in F^1$ as a vertex of $G_F^0$. Write $q_{F^0}:=\sum_{[v]\in F^0}q_{[v]}$. We compute
\begin{enumerate}[(i)]
\item
$$q_{r(e_i)} \sum_{\substack{f\in F^1 \\ s(f)\subseteq r(e_i)}}Q_f=q_{r(e_i)}\sum_{\substack{f\in F^1 \\ s(f)\subseteq r(e_i)}} t_f t_f^*=q_{r(e_i)}\sum_{f\in F^1}t_f t_f^*;$$
\item
\begin{align*}
q_{r(e_i)}\sum_{\substack {[v]\in F^0,\\ [v]\subseteq r(e_i)}}Q_{[v]}&=q_{r(e_i)}\sum_{[v]\in F^0}q_{[v]}(1-\sum_{e\in F^1}t_e t_e^*)\\
&=q_{r(e_i)} q_{F^0}(1-\sum_{e\in F^1}t_e t_e^*);
\end{align*}
\item
$$\sum_{\omega\in \Gamma,\omega_i=1}Q_\omega =\sum_{\omega\in \Gamma,\omega_i=1}q_{R(\omega)}(1-\sum_{e\in F^1}t_e t_e^*)=\sum_{\omega_i=1}q_{R(\omega)}(1-\sum_{e\in F^1}t_e t_e^*),$$
because $\sum_{\omega_i=1}q_{R(\omega)}=q_{r(e_i)}(1-q_{F^0})$.
\end{enumerate}
We can use these relations to get
\begin{align*}
\sum_{s(f)\subseteq r(e_i)}&T_{(e_i,f)}+\sum_{[v]\in F^0,~[v]\subseteq r(e_i)}T_{(e_i,[v])}+ \sum_{\omega\in \Gamma,~\omega_i=1}T_{(e_i,\omega)}\\
&=t_{e_i} \left(q_{r(e_i)}\sum_{e\in F^1}t_et_e^*+q_{r(e_i)} q_{F^0}(\sum_{e\in F^1}t_et_e^*)+ q_{r(e_i)}(1-q_{F^0})(\sum_{e\in F^1}t_et_e^*) \right)\\
&=t_{e_i}q_{r(e_i)}\left(\sum_{e\in F^1}t_et_e^*+(q_{F^0}+1-q_{F^0})(1-\sum_{e\in F^1}t_et_e^*) \right)\\
&=t_{e_i}.
\end{align*}
\begin{flushright}
        (4.1)
\end{flushright}
Now if $e_i$ is not a sink as a vertex in $G_F$ (i.e. $|\{x\in G_F^1:s_F(x)=e_i\}|>0$), we conclude that
\begin{align*}
\sum_{f\in F^1,~s(f)\subseteq r(e_i)}&T_{(e_i,f)}T_{(e_i,f)}^*+\sum_{[v]\in F^0,~[v]\subseteq r(e_i)}T_{(e_i,[v])}T_{(e_i,[v])}^*+ \sum_{\omega\in \Gamma,~\omega_i=1}T_{(e_i,\omega)}T_{(e_i,\omega)}^*\\
&=\sum t_{e_i}Q_f t_{e_i}^*+\sum t_{e_i}Q_{[v]} t_{e_i}^*+\sum t_{e_i}Q_\omega t_{e_i}^*\\
&=t_{e_i}q_{r(e_i)}(\sum Q_f +\sum Q_{[v]}+\sum Q_\omega)t_{e_i}^*\\
&=t_{e_i}t_{e_i}^*=Q_{e_i},
\end{align*}
which establishes the relation (GA3).

Furthermore, equation (4.1) in above says that $t_{e_i}\in C^*(T_*,Q_*)$ for every $e_i\in F^1$. Also, for each $[v]\in F^0$, we have
\begin{align*}
Q_{[v]}+\sum_{e\in F^1, s(e)=[v]} Q_e&=t_{[v]}(1-\sum_{e\in F^1}t_et_e^*)+\sum_{e\in F^1, s(e)=[v]}t_e t_e^*\\
&=t_{[v]}-t_{[v]}\sum_{e\in F^1} t_et_e^*+t_{[v]}\sum_{e\in F^1} t_et_e^*\\
&=t_{[v]}.
\end{align*}
Therefore, the family $\{T_x,Q_a:a\in G_F^0, x\in G_F^1\}$ generates the $C^*$-subalgebra $C^*(\{t_e,q_{[v]}:e\in F^1, [v]\in F^0\})$ of $C^*(\gh)$ and the proof is complete.
\end{proof}

\begin{cor}\label{cor4.3}
If $F$ is a finite subset of $\sg \cup \pgl$, then $C^*(G_F)$ is isometrically isomorphic to the $C^*$-subalgebra of $C^*(\gh)$ generated by $\{t_e,q_{[v]}: [v]\in F^0, e\in F^1\}$.
\end{cor}

\begin{proof}
Suppose that $X$ is the $C^*$-subalgebra generated by $\{t_e,q_{[v]}: [v]\in F^0, e\in F^1\}$ and let $\{T_x,Q_a: a\in G_F^0,x\in G_F^1\}$ be the Cuntz-Krieger $G_F$-family in Proposition \ref{prop4.2}. If $C^*(G_F)=C^*(s_x,p_a)$, then there exists a $*$-homomorphism $\phi:C^*(G_F)\rightarrow X$ with $\phi(p_a)=Q_a$ and $\phi(s_x)=T_x$ for every $a\in G_F^0$, $x\in G_F^1$. Since each $Q_a$ is nonzero by Proposition \ref{prop4.2}, the gauge invariant uniqueness theorem implies that $\phi$ is injective. Moreover, the family $\{T_x,Q_a\}$ generates $X$, so $\phi$ is an isomorphism.
\end{proof}

Note that if $F_1\subseteq F_2$ are two finite subsets of $\sg \cup \pgl$ and $X_1,X_2$ are the $C^*$-subalgebras of $C^*(\gh)$ associated to $G_{F_1}$ and $G_{F_2}$, respectively, we then have $X_1\subseteq X_2$ by Proposition \ref{prop4.2}.

\begin{rem}
Using relations (QA1)-(QA4) in Definition \ref{defn3.8}, each $q_{[A]}$ for $[A]\in \pGo$, can be produced by the elements of
$$\{q_{[v]}:[v]\in \sg\} \cup \{t_e:e\in \pgl\}$$
with finitely many operations. So, the $*$-subalgebra of $C^*(\gh)$ generated by
$$\{q_{[v]}:[v]\in \sg\} \cup \{t_e:e\in \pgl\}$$
is dense in $C^*(\gh)$.
\end{rem}

As for graph $C^*$-algebras, we can apply the universal property to have a strongly continuous \emph{gauge action} $\gamma:\mathbb{T}\rightarrow \mathrm{Aut}(C^*(\gh))$ such that
$$\gamma_z(t_e)=zt_e ~~~ \mathrm{and} ~~~ \gamma_z(q_{[A]})=q_{[A]}$$
for every $[A]\in \pgo$, $e\in \pgl$, and $z\in\mathbb{T}$. Now we are ready to prove the uniqueness theorems.

\begin{thm}[The Gauge Invariant Uniqueness Theorem]\label{thm4.5}
Let $\gh$ be a quotient ultragraph and let $\{T_e,Q_{[A]}\}$ be a representation for $\gh$ such that $Q_{[A]}\neq 0$ for $[A]\neq [\emptyset]$. If $\pi_{T,Q}:C^*(\gh)\rightarrow C^*(T_e,Q_{[A]})$ is the $*$-homomorphism satisfying $\pi_{T,Q}(t_e)=T_e$, $\pi_{T,Q}(q_{[A]})=Q_{[A]}$, and there is a strongly continuous action $\beta$ of $\mathbb{T}$ on $C^*(T_e,Q_{[A]})$ such that $\beta_z\circ \pi_{T,Q}=\pi_{T,Q}\circ \gamma_z$ for every $z\in \mathbb{T}$, then $\pi_{T,Q}$ is faithful.
\end{thm}

\begin{proof}
Select an increasing sequence $\{F_n\}$ of finite subsets of $\sg \cup \pgl$ such that $\cup_{n=1}^\infty F_n=\sg\cup \pgl$. For each $n$, Corollary \ref{cor4.3} gives an isomorphism
$$\pi_n:C^*(G_{F_n})\rightarrow C^*(\{t_e,q_{[v]}:[v]\in F^0,e\in F^1\})$$
that respects the generators. We can apply the gauge invariant uniqueness theorem for graph $C^*$-algebras to see that the homomorphism
$$\pi_{T,Q}\circ \pi_n:C^*(G_{F_n})\rightarrow C^*(T_e,Q_{[A]})$$
is faithful. Hence, for every $F_n$, the restriction of $\pi_{T,Q}$ on the $*$-subalgebra of $C^*(\gh)$ generated by $\{t_e,q_{[v]}:[v]\in F_n^0, e\in F_n^1\}$ is faithful. This turns out that $\pi_{T,Q}$ is injective on the $*$-subalgebra $C^*(t_e,q_{[v]}:[v]\in \sg, e\in \pgl)$. Since, this subalgebra is dense in $C^*(\gh)$, we conclude that $\pi_{T,Q}$ is faithful.
\end{proof}

\begin{prop}\label{prop5.1}
Let $\g$ be an ultragraph. If $(H,B)$ is an admissible pair in $\g$, then $C^*(\gh)\cong C^*(\g)/I_{(H,B)}$.
\end{prop}

\begin{proof}
Using Proposition \ref{prop3.2}, we can consider $I_{(H,B)}$ as an ideal of $C^*(\overline{\g})$. Suppose that $C^*(\overline{\g})=C^*(s_e,p_A)$ and $C^*(\gh)=C^*(t_e,q_{[A]})$. If we define
$$T_e:=s_e+I_{(H,B)} ~~~ \mathrm{and} ~~~ Q_{[A]}:=p_A+I_{(H,B)}$$
for every $[A]\in \pgo$ and $e\in \pgl$, then the family $\{T_e,Q_{[A]}\}$ is a representation for $\gh$ in $C^*(\overline{\g})/I_{(H,B)}$. So, there is a $*$-homomorphism $\phi:C^*(\gh)\rightarrow C^*(\g)/I_{(H,B)}$ such that $\phi(t_e)=T_e$ and $\phi(q_{[A]})=Q_{[A]}$. Moreover, all $Q_{[A]}$ with $[A]\neq [\emptyset]$ are nonzero because $p_A+I_{(H,B)}=I_{(H,B)}$ implies $A\in H$. Then, an application of Theorem \ref{thm4.5} yields that $\phi$ is faithful. On the other hand, the family $\{T_e,Q_{[A]}:[A]\in\pgo,e\in\pgl\}$ generates the quotient $C^*(\g)/I_{(H,B)}$, and hence, $\phi$ is surjective as well. Therefore, $\phi$ is an isomorphism and the result follows.
\end{proof}

To prove a version of Cuntz-Krieger uniqueness theorem, we extend Condition (L) for quotient ultragraphs.

\begin{defn}
We say that $\gh$ satisfies \emph{Condition (L)} if for every loop $\alpha=e_1\ldots e_n$ in $\gh$, at least one of the following conditions holds:
\begin{enumerate}[(i)]
  \item $r(e_i)\neq s(e_{i+1})$ for some $1\leq i\leq n$, where $e_{i+1}:=e_1$ (or equivalently, $r(e_i)\setminus s(e_{i+1})\ne [\emptyset]$).
  \item $\alpha$ has an exit; that means, there exists $f\in \pgl$ such that $s(f)\subseteq r(e_i)$ and $f\neq e_{i+1}$ for some $1\leq i\leq n$.
\end{enumerate}
\end{defn}

\begin{lem}\label{lem4.7}
Let $F$ be a finite subset of $\sg \cup \pgl$. If $\gh$ satisfies Condition (L), so does the graph $G_F$.
\end{lem}

\begin{proof}
Suppose that $\gh$ satisfies Condition (L). As the elements of $F^0\cup \Gamma$ are sinks in $G_F$, every loop in $G_F$ is of the form $\widetilde{\alpha}=(e_1,e_2)\ldots (e_n, e_1)$ corresponding with a loop $\alpha=e_1\ldots e_n$ in $\gh$. So, fix a loop $\widetilde{\alpha}=(e_1,e_2)\ldots (e_n, e_1)$ in $G_F$. Then $\alpha=e_1\ldots e_n$ is a loop in $\gh$ and by Condition (L), one of the following holds:
\begin{enumerate}[(i)]
\item $r(e_i)\neq s(e_{i+1})$ for some $1\leq i\leq n$, where $e_{i+1}:=e_1$, or
\item there exists $f\in \pgl$ such that $s(f)\subseteq r(e_i)$ and $f\neq e_{i+1}$ for some $1\leq i\leq n$.
\end{enumerate}

We can suppose in the case (i) that $s(e_{i+1})\subsetneq r(e_i)$ and $r(e_i)$ emits only the edge $e_{i+1}$ in $\gh$. Then, by the definition of $\Gamma$, there exists either $[v]\in F^0$ with $[v]\subseteq r(e_i)\setminus s(e_{i+1})$, or $\omega \in \Gamma $ with $\omega_i=1$. Thus, either $(e_i,[v])$ or $(e_i,\omega)$ is an exit for the loop $\widetilde{\alpha}$ in $G_F$, respectively.

Now assume case (ii) holds. If $f\in F^1$, then $(e_i,f)$ is an exit for $\widetilde{\alpha}$. If $f\notin F^1$, for $[v]:=s(f)$ we have either $[v]\notin F^0$ or
$$\exists \omega\in \Gamma ~ \mathrm{with} ~ \omega_i=1 ~ \mathrm{such~that} ~ [v]\subseteq R(\omega).$$
Hence, $(e_i,[v])$ or $(e_i,\omega)$ is an exit for $\widetilde{\alpha}$, respectively. Consequently, in any case, $\widetilde{\alpha}$ has an exit.
\end{proof}

\begin{thm}[The Cuntz-Krieger Uniqueness Theorem]\label{thm4.8}
Suppose that $\gh$ is a quotient ultragraph satisfying Condition (L). If $\{T_e,Q_A\}$ is a Cuntz-Krieger representation for $\gh$ in which all the projection $Q_{[A]}$ are nonzero for $[A]\neq [\emptyset]$, then the $*$-homomorphism $\pi_{T,Q}:C^*(\gh)\rightarrow C^*(T_e,Q_{[A]})$ with $\pi_{T,Q}(t_e)=T_e$ and $\pi_{T,Q}(q_{[A]})=Q_{[A]}$ is an isometrically isomorphism.
\end{thm}

\begin{proof}
It suffices to show that $\pi_{T,Q}$ is faithful. Similar to Theorem \ref{thm4.5}, choose an increasing sequence $\{F_n\}$ of finite sets such that $\cup_{n=1}^\infty F_n=\sg \cup \pgl$. By Corollary \ref{cor4.3}, there are isomorphisms $\pi_n:  C^*(G_{F_n}) \rightarrow C^*(\{t_e,q_{[v]}:[v]\in F_n^0,e\in F_n^1\})$ that respect the generators. Since all the graphs $G_{F_n}$ satisfy Condition (L) by Lemma \ref{lem4.7}, the Cuntz-Krieger uniqueness theorem for graph $C^*$-algebras implies that the $*$-homomorphisms
$$\pi_{T,Q}\circ \pi_n: C^*(G_{F_n})\rightarrow C^*(T_e,Q_{[A]})$$
are faithful. Therefore, $\pi_{T,Q}$ is faithful on the subalgebra $C^*(t_e,q_{[v]}:[v]\in \sg, e\in \pgl)$ of $C^*(\gh)$. Since this subalgebra is dense in $C^*(\gh)$, we conclude that $\pi_{T,Q}$ is a faithful homomorphism.
\end{proof}

%%%%%%%%%%%%%%%%%%%%%%%%%%%%%%%%%%%%%%%%%%%%%%%%

\section{Primitive ideals in $C^*(\g)$}

In this section, we apply quotient ultragraphs to describe primitive gauge invariant ideals of an ultragraph $C^*$-algebra. Recall that since every ultragraph $C^*$-algebra $C^*(\g)$ is separable (as assumed $\go$ to be countable), a prime ideal of $C^*(\g)$ is primitive and vice versa \cite[Corollaire 1]{dix}.

To prove Proposition \ref{prop7.3} below, we need the following simple lemmas.

\begin{lem}\label{lem5.1}
Let $\g/(H,B)=(\pGo,\pgo, \pgl,r,s)$ be a quotient ultragraph of $\g$. If $\gh$ does not satisfy Condition (L), then $C^*(\g/(H,B))$ contains an ideal Morita-equivalent to $C(\mathbb{T})$.
\end{lem}

\begin{proof}
Suppose that $\gamma=e_1\ldots e_n$ is a loop in $\g/(H,B)$ without exits and $r(e_i)=s(e_{i+1})$ for $1\leq i\leq n$. If $C^*(\gh)=C^*(t_e,q_{[A]})$, for each $i$ we have
$$t_{e_i}^*t_{e_i}=q_{r(e_i)}=q_{s(e_{i+1})}=t_{e_{i+1}}t_{e_{i+1}}^*.$$
Write $[v]:=s(\gamma)$ and let $I_{\gamma}$ be the ideal of $C^*(\g/(H,B))$ generated by $q_{[v]}$. Since $\gamma$ has no exits in $\gh$ and we have
$$q_{s(e_i)}=(t_{e_i}\ldots t_{e_n})q_{[v]}(t_{e_n}^* \ldots t_{e_i}^*)  \hspace{10mm} (1\leq i\leq n),$$
an easy argument shows that
$$I_{\gamma}=\overline{\mathrm{span}}\left\{t_{\alpha}q_{[v]}t_{\beta}^*: \alpha,\beta\in (\g/(H,B))^*, [v]\subseteq r(\alpha)\cap r(\beta)\right\}.$$
So, we get
$$q_{[v]}I_{\gamma}q_{[v]}=\overline{\mathrm{span}}\left\{(t_{\gamma})^n q_{[v]} (t_{\gamma}^*)^m: m,n\geq 0\right\},$$
where $(t_\gamma)^0=(t_\gamma^*)^0:=q_{[v]}$. We show that $q_{[v]}I_{\gamma}q_{[v]}$ is a full corner in $I_{\gamma}$ which is isometrically isomorphic to $C(\mathbb{T})$. For this, let $E$ be the graph with one vertex $w$ and one loop $f$. If we set $Q_w:=q_{[v]}$ and $T_f:=t_{\gamma}~(=t_{\gamma}q_{[v]})$, then $\{T_f,Q_w\}$ is a Cuntz-Krieger $E$-family in $q_{[v]}I_{\gamma}q_{[v]}$. Assume $C^*(E)=C^*(s_f,p_w)$. Since $Q_w\neq 0$, the gauge-invariant uniqueness theorem for graph $C^*$-algebras implies that the $*$-homomorphism $\phi:C^*(E)\rightarrow q_{[v]}I_{\gamma}q_{[v]}$ with $p_w\mapsto Q_w$ and $s_f\mapsto T_f$ is faithful. Moreover, the $C^*$-algebra $q_{[v]}I_{\gamma}q_{[v]}$ is generated by $\{T_f,Q_w\}$, and hence $\phi$ is an isomorphism. As we know $C^*(E)\cong C(\mathbb{T})$, $q_{[v]}I_{\gamma}q_{[v]}$ is isomorphic to $C(\mathbb{T})$. Moreover, since $q_{[v]}$ generates $I_\gamma$, the corner $q_{[v]}I_{\gamma}q_{[v]}$ is full in $I_\gamma$. Thus, $I_\gamma$ is Morita-equivalent to $q_{[v]}I_\gamma q_{[v]}\cong C(\mathbb{T})$ and the proof is complete.
\end{proof}

\begin{lem}\label{lem7.2}
If $\gh$ satisfies Condition (L), then any nonzero ideal in $C^*(\gh)$ contains projection $q_{[A]}$ for some $[A]\neq [\emptyset]$.
\end{lem}

\begin{proof}
Take an arbitrary ideal $J$ in $C^*(\gh)$. If there are no $q_{[A]}\in J$ with $[A]\neq[\emptyset]$, then Theorem \ref{thm4.8} implies that the quotient homomorphism $\phi:C^*(\gh)\rightarrow C^*(\gh)/J$ is injective. Hence, we have $J=\ker \phi=(0)$.
\end{proof}

\begin{defn}\label{defn7.1}
Let $\g$ be an ultragraph. For two sets $A,B\in \go$, we write $A\geq B$ if either $B\subseteq A$, or there exists $\alpha\in \g^*$ with $|\alpha|\geq 1$ such that $s(\alpha)\in A$ and $B\subseteq r(\alpha)$. We simply write $A\geq v$, $v\geq B$, and $v\geq w$ if $A\geq \{v\}$, $\{v\}\geq B$, and $\{v\}\geq \{w\}$, respectively. A subset $M\subseteq \go$ is said to be \emph{downward directed} whenever for every $A,B\in M$, there exists $\emptyset\neq C\in M$ such that $A,B\geq C$.
\end{defn}

\begin{prop}\label{prop7.3}
Let $H$ be a saturated hereditary subset of $\go$. Then the ideal $I_{(H,B_H)}$ in $C^*(\g)$ is primitive if and only if the quotient ultragraph $\g/(H,B_H)$ satisfies Condition (L) and the collection $\go \setminus H$ is downward directed.
\end{prop}

\begin{proof}
Let $I_{(H,B_H)}$ be a primitive ideal of $C^*(\g)$. Since $C^*(\g)/I_{(H,B_H)}\cong C^*(\g/(H,B_H))$, the zero ideal in $C^*(\g/(H,B_H))$ is primitive. If $\g/(H,B_H)$ does not satisfy Condition (L), then $C^*(\g/(H,B_H))$ contains an ideal $J$ Morita-equivalent to $C(\mathbb{T})$  by Lemma \ref{lem5.1}. Select two ideals $I_1,I_2$ in $C(\mathbb{T})$ with $I_1\cap I_2=(0)$, and let $J_1,J_2$ be their corresponding ideals in $J$. Then $J_1$ and $J_2$ are two nonzero ideals of $C^*(\g/(H,B_H))$ with $J_1\cap J_2=(0)$, contradicting the primness of $C^*(\g/(H,B_H))$. Therefore, $\gh$ satisfies Condition (L).

Now we show that $M:=\go\setminus H$ is downward directed. For this, we take two arbitrary sets $A,B\in M$ and consider the ideals
$$J_1:=C^*(\g/(H,B_H)) q_{[A]} C^*(\g/(H,B_H))$$
and
$$J_2:=C^*(\g/(H,B_H)) q_{[B]} C^*(\g/(H,B_H))$$
in $C^*(\g/(H,B_H))$ generated by $q_{[A]}$ and $q_{[B]}$, respectively. Since $A,B\notin H$, the projections $q_{[A]},q_{[B]}$ are nonzero by Theorem \ref{thm3.11}, and so are the ideals $J_1,J_2$. The primness of $C^*(\g/(H,B_H))$ implies that the ideal
$$J_1J_2=C^*\left(\g/(H,B_H)\right) q_{[A]} C^*\left(\g/(H,B_H)\right) q_{[B]} C^*\left(\g/(H,B_H)\right)$$
is nonzero, and hence $q_{[A]} C^*(\g/(H,B_H)) q_{[B]} \neq \{0\}$. As the set
$$\mathrm{span}\left\{ t_\alpha q_{[D]} t_\beta^* : \alpha, \beta \in (\gh)^*,~ r(\alpha)\cap [D]\cap r(\beta)\neq [\emptyset] \right\}$$
is dense in $C^*(\g/(H,B_H))$, there exist $\alpha,\beta\in (\g/(H,B_H))^*$ and $[D]\in \pgo$ such that $q_{[A]}(t_\alpha q_{[D]} t_\beta^*) q_{[B]}\neq 0$. In this case, we must have $s(\alpha)\subseteq [A]$ and $s(\beta)\subseteq [B]$ and thus, $A,B\geq C$ for $C:=r_\g(\alpha) \cap D \cap r_\g(\beta)$. Therefore, $\go\setminus H$ is downward directed.

For the converse, we assume that $\g/(H,B_H)$ satisfies Condition (L) and the collection $M=\go\setminus H$ is downward directed. Fix two nonzero ideals $J_1,J_2$ of $C^*(\g/(H,B_H))$. By Lemma \ref{lem7.2}, there are nonzero projections $q_{[A]}\in J_1$ and $q_{[B]}\in J_2$. Then $A,B\notin H$ and, since $M$ is downward directed, there exists $C\in M$ such that $A,B\geq C$. Hence, the ideal $J_1\cap J_2$ contains the nonzero projection $q_{[C]}$. Since $J_1$ and $J_2$ were arbitrary, this concludes that the $C^*$-algebra $C^*(\g/(H,B_H))$ is primitive and $I_{(H,B_H)}$ is a primitive ideal in $C^*(\g)$ by Proposition \ref{prop5.1}.
\end{proof}

The following proposition describes another kind of primitive ideals in $C^*(\g)$.

\begin{prop}\label{prop7.4}
Let $(H,B)$ be an admissible pair in $\g$ and let $B=B_H\setminus \{w\}$. Then the ideal $I_{(H,B)}$ in $C^*(\g)$ is primitive if and only if $A\geq w$ for all $A\in \go\setminus H$.
\end{prop}

\begin{proof}
Suppose that $I_{(H,B)}$ is a primitive ideal and take an arbitrary $A\in \go\setminus H$. If $\overline{A}:=A\cup \{v':v\in A\cap (B_H\setminus B)\}$, then $q_{[\overline{A}]}$ and $q_{[w']}$ are two nonzero projections in $C^*(\g/(H,B))$. If we consider ideals $J_{[\overline{A}]}:=\langle q_{[\overline{A}]} \rangle$ and $J_{[w']}:= \langle q_{[w']} \rangle$ in $C^*(\gh)$, then the primness of $C^*(\gh)\cong C^*(\g)/I_{H,B}$ implies that the ideal
$$J_{[\overline{A}]} J_{[w']}=C^*(\g/(H,B)) q_{[\overline{A}]} C^*(\g/(H,B)) q_{[w']} C^*(\g/(H,B))$$
is nonzero, and hence $q_{[\overline{A}]} C^*(\g/(H,B)) q_{[w']}\neq \{0\}$. So, there exist $\alpha,\beta\in (\g/(H,B))^*$ such that $q_{[\overline{A}]}t_\alpha t_\beta^* q_{[w']}\neq 0$. Since $[w']$ is a sink in $\g/(H,B)$, we must have $q_{[\overline{A}]} t_\alpha q_{[w']}\neq 0$. If $|\alpha|=0$, then $[w']\subseteq [\overline{A}]$, $w'\in \overline{A}$ and $w\in A$. If $|\alpha|\geq 1$, then $s(\alpha)\subseteq [\overline{A}]$ and $[w']\subseteq r(\alpha)$, which follow $s_\g(\alpha)\in A$ and $w\in r_\g(\alpha)$. Therefore, we obtain $A\geq w$ in either case.

Conversely, assume $A\geq w$ for every $A\in \go\setminus H$. Then the collection $\go\setminus H$ is downward directed. Moreover, for every $[\emptyset]\ne [A]\in \pgo$, there exists $\alpha\in (\gh)^*$ such that $s(\alpha)\subseteq [A]$ and $[w']\subseteq r(\alpha)$. As $[w']$ is a sink in $\gh$, we see that the quotient ultragraph $\gh$ satisfies Condition (L). Now similar to the proof of Proposition \ref{prop7.3}, we can show that $I_{(H,B)}$ is a primitive ideal.
\end{proof}

Recall that each loop in $\gh$ comes from a loop in the initial ultragraph $\g$. So, to check Condition (L) for a quotient ultragraph $\gh$, we can use the following.

\begin{defn}
Let $H$ be a saturated hereditary subset of $\go$. For simplicity, we say that a path $\alpha=e_1\ldots e_n$ lies in $\g\setminus H$ whenever $r_\g(\alpha)\in \go\setminus H$. We also say that $\alpha$ has \emph{an exit in $\g\setminus H$} if either $r_\g(e_i)\setminus s_\g(e_{i+1})\in \go\setminus H$ for some $i$, or there is an edge $f$ with $r_\g(f)\in \go\setminus H$ such that $s_\g(f)=s_\g(e_i)$ and $f\neq e_i$, for some $1\leq i\leq n$.
\end{defn}

It is easy to verify that a quotient ultragraph $\g/(H,B)$ satisfies Condition (L) if and only if every loop in $\g\setminus H$ has an exit in $\g\setminus H$. Hence we have:
\begin{thm}\label{thm7.6}
Let $\g$ be an ultragraph. A gauge invariant ideal $I_{(H,B)}$ of $C^*(\g)$ is primitive if and only if one of the following holds:
\begin{enumerate}
\item $B=B_H$, $\go\setminus H$ is downward directed, and every loop in $\g\setminus H$ has an exit in $\g\setminus H$.
\item $B=B_H\setminus\{w\}$ for some $w\in B_H$, and $A\geq w$ for all $A\in \go\setminus H$.
\end{enumerate}
\end{thm}

\begin{proof}
Let $I_{(H,B)}$ be a primitive ideal in $C^*(\g)$. Then $C^*(\gh)\cong C^*(\g)/I_{(H,B)}$ is a primitive $C^*$-algebra. We claim that $|B_H\setminus B|\leq 1$. Indeed, if $w_1,w_2$ are two distinct vertices in $B_H\setminus B$, similar to the proof of Propositions \ref{prop7.3} and \ref{prop7.4}, the primitivity of $C^*(\gh)$ implies that the corner $q_{[w_1']} C^*(\gh) q_{[w_2']}$ is nonzero. So, there exist $\alpha,\beta\in (\gh)^*$ such that $q_{[w_1']}t_\alpha t_\beta^* q_{[w_2']}\neq 0$. But we must have $|\alpha|=|\beta|=0$ because $[w'_1],[w'_2]$ are two sinks in $\gh$.  Hence, $q_{[w_1']}q_{[w_2']}\neq 0$ which is impossible because $q_{[w_1']}q_{[w_2']}=q_{[\{w_1'\}\cap \{w_2'\}]}=q_{[\emptyset]}=0$. Thus, the claim holds. Now we may apply Propositions \ref{prop7.3} and \ref{prop7.4} to obtain the result.
\end{proof}

Following \cite[Definition 7.1]{kat3}, we say that an ultragraph $\g$ satisfies \emph{Condition (K)} if every vertex $v\in G^0$ either is the base of no loops, or there are at least two loops $\alpha,\beta$ in $\g$ based at $v$ such that neither $\alpha$ nor $\beta$ is a subpath of the other. In view of \cite[Proposition 7.3]{kat3}, if $\g$ satisfies Condition (K), then all ideals of $C^*(\g)$ are of the form $I_{(H,B)}$. So, in this case, Theorem \ref{thm7.6} describes all primitive ideals of $C^*(\g)$.

%%%%%%%%%%%%%%%%%%%%%%%%%%%%%%%%%%%%%%%%%%%%%%%%

\section{Purely infinite ultragraph $C^*$-algebras via Fell bundles}

Mark Tomforde in \cite{tom2} determined ultragraph $C^*$-algebras in which every hereditary subalgebra contains infinite projections. Here, we consider the notion of ``\emph{pure infiniteness}'' in the sense of Kirchberg-R${\o}$rdam \cite{kirch00}. In view of Proposition 3.14 and Theorem 4.16 of \cite{kirch00}, a (not necessarily simple) $C^*$-algebra $A$ is \emph{purely infinite} if and only if for every $a\in A^+\setminus \{0\}$ and closed two-sided ideal $I\trianglelefteq A$, $a+I$ in the quotient $A/I$ is either zero or infinite (in this case, $a$ is called \emph{properly infinite}). Recall from \cite[Definition 3.2]{kirch00} that an element $a\in A^+\setminus \{0\}$ is called {\it infinite} if there is $b\in A^+\setminus \{0\}$ such that $a\oplus b\lesssim a\oplus 0$ in the matrix algebra $M_2(A)$.

So, the notion of pure infiniteness is directly related to the structure of ideals and quotients. In this section, we use the quotient ultragraphs to characterize purely infinite ultragraph $C^*$-algebras. Briefly, we consider the natural $\mathbb{Z}$-grading (or Fell bundle) for $C^*(\g)$ and then apply the results of \cite[Section 4]{kwa17} for pure infiniteness of Fell bundles.

\subsection{Condition (K) for $\g$}

To prove the main result of this section, Theorem \ref{thm6.7}, we need to show that an ultragraph $\g$ satisfies Condition (K) if and only if every quotient ultragraph $\gh$ satisfies Condition (L).

{\bf Notation.} Let $\alpha=e_1\ldots e_n$ be a path in an ultragraph $\g$. If $\beta=e_k e_{k+1} \ldots e_l$ is a subpath of $\alpha$, we simply write $\beta\subseteq \alpha$; otherwise, we write $\beta\nsubseteq \alpha$.

First, we show in the absence of Condition (K) for $\g$ that there is a quotient ultragraph $\g/(H,B)$ which does not satisfy Condition (L). For this, let $\g$ contain a loop $\gamma=e_1\ldots e_n$ such that there are no loops $\alpha$ with $s(\alpha)=s(\gamma)$, $\alpha\nsubseteq \gamma$, and $\gamma\nsubseteq\alpha$. If $\gamma^0:=\{s_\g(e_1),\ldots, s_\g(e_n)\}$, define
$$X:=\left\{r_\g(\alpha)\setminus \gamma^0: \alpha\in \g^*,|\alpha|\geq 1, s_\g(\alpha)\in \gamma^0\right\},$$
$$Y:=\left\{\bigcup_{i=1}^n A_i: A_1,\ldots,A_n\in X , n\in \mathbb{N}\right\},$$
and set
$$H_0:=\left\{B\in \go:B\subseteq A ~ \mathrm{for ~ some} ~ A\in Y\right\}.$$
We construct a saturated hereditary subset $H$ of $\go$ as follows: for any $n\in \mathbb{N}$ inductively define
$$S_n:=\left\{w\in G^0: 0<|s_\g^{-1}(w)|<\infty ~~ \mathrm{and} ~~ r_\g(s_\g^{-1}(w))\subseteq H_{n-1}\right\}$$
and
$$H_n:=\left\{A\cup F: A\in H_{n-1} ~~ \mathrm{and} ~~ F\subseteq S_n ~~\mathrm{is ~ a ~ finite ~ subset} \right\}.$$
Then we can see that the subset
$$H=\bigcup_{n=0}^\infty H_n=\left\{A\cup F: A\in H_0 ~~ \mathrm{and} ~~ F\subseteq \bigcup_{n=1}^\infty S_n ~~ \mathrm{is~a ~ finite ~ subset}\right\}$$
is hereditary and saturated.

\begin{lem}\label{lem6.3}
Suppose that $\gamma=e_1\ldots e_n$ is a loop in $\g$ such that there are no loops $\alpha$ with $s(\alpha)=s(\gamma)$ and $\alpha\nsubseteq \gamma$, $\gamma\nsubseteq\alpha$. If we construct the set $H$ as above, then $H$ is a saturated hereditary subset of $\go$. Moreover, we have $A\cap \gamma^0=\emptyset$ for every $A\in H$.
\end{lem}

\begin{proof}
By induction, we first show that each $H_n$ is a hereditary set in $\g$. For this, we check conditions (H1)-(H3) in Definition \ref{defn2.5}. To verify condition (H1) for $H_0$, let us take $e\in \gl$ with $s_\g(e)\in H_0$. Then $s_\g(e)\in X$ and there is $\alpha\in \g^*$ such that $s_\g(\alpha)\in \gamma^0$ and $s_\g(e)\in r_\g(\alpha)\setminus \gamma^0$. Hence, $s_\g(\alpha e)=s_\g(\alpha)\in \gamma^0$. Moreover, we have $r_\g(\alpha e)\cap \gamma^0=\emptyset$ because the otherwise implies the existence of a path $\beta\in \g^*$ with $s_\g(\beta)=s_\g(\gamma)$ and $\beta\nsubseteq \gamma$, $\gamma\nsubseteq \beta$, contradicting the hypothesis. It turns out
$$r_\g(e)=r_\g(\alpha e)=r_\g(\alpha e)\setminus \gamma^0 \in X\subseteq H_0.$$
Hence, $H_0$ satisfies condition (H1). We may easily verify conditions (H2) and (H3) for $H_0$, so $H_0$ is hereditary. Moreover, for every $w\in S_n$, the range of each edge emitted by $w$ belongs to $H_{n-1}$ by definition. Thus, we can inductively check that each $H_n$ is hereditary, and so is $H=\cup_{n=1}^\infty H_n$. The saturation property of $H$ may be verified similar to the proof of \cite[Lemma 3.12]{tom2}.

It remains to show $A\cap \gamma^0=\emptyset$ for every $A\in H$. To do this, note that $A\cap \gamma^0=\emptyset$ for every $A\in H_0$ because this property holds for all $A\in X$. We claim that $(\cup_{n=1}^\infty S_n)\cap \gamma^0=\emptyset$. Indeed, if $v=s_\g(e_i)\in \gamma^0$ for some $e_i\in \gamma$, then $r_\g(e_i)\cap \gamma^0\neq \emptyset$ and $r_\g(e_i)\notin H_0$. Hence, $\{r_\g(e):e\in \gl,~s_\g(e)=v\}\nsubseteq H_0$ that turns out $v\notin S_1$. So, we have $S_1\cap \gamma^0=\emptyset$. An inductive argument shows $S_n\cap \gamma^0=\emptyset$ for $n\geq 1$, and the claim holds. Now since
$$H=\cup_{n=1}^\infty H_n=\left\{A\cup F: A\in H_0 ~~\mathrm{and} ~~ F\subseteq \cup_{n=1}^\infty S_n ~~ \mathrm{is~ a ~ finite ~ subset}\right\},$$
we conclude that $A\cap \gamma^0=\emptyset$ for all $A\in H$.
\end{proof}

\begin{prop}\label{prop6.3}
An ultragraph $\g$ satisfies Condition (K) if and only if for every admissible pair $(H,B)$ in $\g$, the quotient ultragraph $\g/(H,B)$ satisfies Condition (L).
\end{prop}

\begin{proof}
Suppose that $\g$ satisfies Condition (K) and $(H,B)$ is an admissible pair in $\g$. Let $\alpha=e_1\ldots e_n$ be a loop in $\g/(H,B)$. Since $\alpha$ is also a loop in $\g$, there is a loop $\beta=f_1\ldots f_m$ in $\g$ with $s_\g(\alpha)=s_\g(\beta)$, and neither $\alpha\subseteq \beta$ nor $\beta\subseteq \alpha$. Without loos of generality, assume $e_1\neq f_1$. By the fact $s_\g(\alpha)=s_\g(\beta)\in r_\g(\beta)$, we have $r_\g(\beta)\notin H$, and so $r_\g(f_1)\notin H$ by the hereditary property of $H$. Therefore, $f_1$ is an exit for $\alpha$ in $\g/(H,B)$ and we conclude that $\g/(H,B)$ satisfies Condition (L).

For the converse, suppose on the contrary that $\g$ does not satisfy Condition (K). Then there exists a loop $\gamma=e_1\ldots e_n$ in $\g$ such that there are no loops $\alpha$ with $s(\alpha)=s(\gamma)$, $\alpha\nsubseteq \gamma$, and $\gamma\nsubseteq \alpha$. As Lemma \ref{lem6.3}, construct a saturated hereditary subset $H$ of $\go$ and consider the quotient ultragraph $\g/(H,B_H)=(\pGo,\pgo,\pgl,r,s)$. We show that $\gamma$ as a loop in $\g/(H,B_H)$ has no exits and $r(e_i)=s(e_{i+1})$ for $1\leq i\leq n$. If $f$ is an exit for $\gamma$ in $\g/(H,B_H)$ such that $s(f)=s(e_j)$ and $f\neq e_j$, then $r_\g(f)\notin H$ and $r_\g(f)\cap \gamma^0\neq \emptyset$ (if $r_\g(f)\cap \gamma^0= \emptyset$, then $r_\g(f)=r_\g(f)\setminus \gamma^0\in X\subseteq H$, a contradiction). So, there is $e_l\in \gamma$ such that $s_\g(e_l)\in r_\g(f)$. If we set $\alpha:=e_1\ldots e_{j-1}fe_l\ldots e_n$, then $\alpha$ is a loop in $\g$ with $s_\g(\alpha)=s_\g(\gamma)$, and $\alpha\nsubseteq \gamma$, $\gamma\nsubseteq \alpha$, that contradicts the hypothesis. Therefore, $\gamma$ has no exits in $\g/(H,B_H)$. Moreover, we have $r(e_i)\cap [\gamma^0]=s(e_{i+1})$ for each $1\leq i\leq n$, because the otherwise gives an exit for $\gamma$ in $\g/(H,B_H)$ by the construction of $H$. Hence,
$$r(e_i)\setminus s(e_{i+1})= r(e_i)\setminus [\gamma^0]=[\emptyset]$$
and we get $r(e_i)=s(e_{i+1})$ (note that the fact $r_\g(e_i)\setminus \gamma^0\in H$ implies $r(e_i)\setminus [\gamma^0 ]=[ \overline{r_\g(e_i)}\setminus \gamma^0 ]=[\emptyset ]$). Therefore, the quotient ultragraph $\g/(H,B_H)$ does not satisfy Condition (L) as desired.
\end{proof}

\subsection{Purely infinite ultragraph $C^*$-algebras via Fell bundles}

Every quotient ultragraph (or ultragraph) $C^*$-algebra $C^*(\gh)=C^*(q_{[A]},t_e)$ is equipped with a natural $\mathbb{Z}$-grading or Fell bundle $\mathcal{B}=\{B_n:n\in \mathbb{Z}\}$ with the fibers
$$B_n:=\overline{\mathrm{span}}\left\{ t_\mu q_{[A]} t_\nu^*:\mu,\nu\in (\gh)^*, ~ |\mu|-|\nu|=n\right\}.$$
These Fell bundles will be considered in this section. The fiber $B_0$ is the fixed point $C^*$-subalgebra of $C^*(\gh)$ for the gauge action which is an AF $C^*$-algebra. An application of the gauge invariant uniqueness theorem implies that $C^*(\gh)$ is isomorphic to the cross sectional $C^*$-algebra $C^*(\mathcal{B})$ (we refer the reader to \cite{exe15} for details about Fell bundles and their $C^*$-algebras). Moreover, since $\mathbb{Z}$ is an amenable group, combining Theorem 20.7 and Proposition 20.2 of \cite{exe15} implies that $C^*(\gh)$ is also isomorphic to the reduced cross sectional $C^*$-algebra $C^*_r(\mathcal{B})$.

Following \cite[Definition 2.1]{exe02}, an \emph{ideal} in a Fell bundle $\mathcal{B}=\{B_n\}$ is a family $\mathcal{J}=\{J_n\}_{n\in\mathbb{Z}}$ of closed subspaces $J_n\subseteq B_n$, such that $B_mJ_n\subseteq J_{mn}$ and $J_n B_m\subseteq J_{nm}$ for all $m,n\in \mathbb{Z}$. If $\mathcal{J}$ is an ideal of $\mathcal{B}$, then the family $\mathcal{B}/\mathcal{J}:=\{B_n/J_n\}_{n\in \mathbb{Z}}$ is equipped with a natural Fell bundle structure, which is called a \emph{quotient Fell bundle} of $\mathcal{B}$, cf. \cite[Definition 21.14]{exe15}.

\begin{defn}[{\cite[Definition 4.1]{kwa17}}]
Let $\gh$ be a quotient ultragraph and $\mathcal{B}=\{B_n\}_{n\in \mathbb{Z}}$ is the above Fell bundle in $C^*(\gh)$. We say that $\mathcal{B}$ is \emph{aperiodic} if for each $n\in\mathbb{Z}\setminus\{0\}$, each $b_n\in B_n$, and every hereditary subalgebra $A$ of $B_0$, we have
$$\inf\left\{\|a b_n a\|: a\in A^+,~ \|a\|=1\right\}=0.$$
Furthermore, $\mathcal{B}$ is called \emph{residually aperiodic} whenever the quotient Fell bundle $\mathcal{B}/\mathcal{J}$ is aperiodic for every ideal $\mathcal{J}$ of $\mathcal{B}$.
\end{defn}

The following lemma is analogous to \cite[Proposition 7.3]{kwa17} for quotient ultragraphs.

\begin{lem}\label{lem6.5}
Let $\gh$ be a quotient ultragraph and let $\mathcal{B}=\{B_n\}_{n\in \mathbb{Z}}$ be the Fell bundle associated to $C^*(\gh)$. Then $\mathcal{B}$ is aperiodic if and only if $\gh$ satisfies Condition (L).
\end{lem}

\begin{proof}
We may modify the proof of \cite[Proposition 7.3]{kwa17} for our case by replacing elements $s_\alpha s_\beta^*$ and $s_\mu s_\mu^*$ with $t_\alpha q_{[A]}t_\beta^*$ and $t_\mu q_{[A]}t_\mu^*$, respectively. Then the proof goes along the same lines as the one in \cite[Proposition 7.3]{kwa17}.
\end{proof}

\begin{cor}\label{cor6.6}
Let $\g$ be an ultragraph and let $\mathcal{B}=\{B_n\}_{n\in\mathbb{Z}}$ be the described Fell bundle of $C^*(\g)$. If $\g$ satisfies Condition (K), then $\mathcal{B}$ is residually aperiodic.
\end{cor}

\begin{proof}
Suppose that $\g$ satisfies Condition (K). In view of \cite[Proposition 7.3]{kat3}, we know that all ideals of $C^*(\g)$ are graded and of the form $I_{(H,B)}$. So, each ideal $\mathcal{J}=\{J_n\}_{n\in\mathbb{Z}}$ of $\mathcal{B}$ is corresponding with an ideal $I_{(H,B)}$ with the homogenous components $J_n:=I_{(H,B)}\cap B_n$. Moreover, the quotient Fell bundle $\mathcal{B}/\mathcal{J}:=\{B_n/J_n:n\in\mathbb{Z}\}$ is a grading (or a Fell bundle) for $C^*(\g)/I_{(H,B)}\cong C^*(\g/(H,B))$. Therefore, quotient Fell bundles $\mathcal{B}/\mathcal{J}$ are corresponding with quotient ultragraphs $\gh$. Since such quotient ultragraphs satisfy Condition (L) by Proposition \ref{prop6.3}, Lemma \ref{lem6.5} follows the result.
\end{proof}

\begin{thm}\label{thm6.7}
Let $\g$ be an ultragraph. Then $C^*(\g)$ is purely infinite (in the sense of \cite{kirch00}) if and only if $\g$ satisfies Condition (K), and for every saturated hereditary subset $H$ of $\go$, we have
\begin{enumerate}
  \item $B_H=\emptyset$, and
  \item every $A\in\go\setminus H$ connects to a loop $\alpha$ in $\g\setminus H$, which means $A\geq s_\g(\alpha)$ (see Definition \ref{defn7.1}).
\end{enumerate}
\end{thm}

\begin{proof}
First, suppose that $C^*(\g)$ is purely infinite. If $\g$ does not satisfy Condition (K), by the second paragraph in the proof of Proposition \ref{prop6.3}, there is a quotient ultragraph $\gh$ containing a loop $\alpha\in (\gh)^*$ with no exits in $\gh$. The argument of Lemma \ref{lem5.1} follows that the ideal $J:=\langle q_{s(\alpha)}\rangle\unlhd C^*(\gh)$ is Morita-equivalent to $C(\mathbb{T})$. Hence, the projection $p_{s(\alpha)}$ is not properly infinite which contradicts \cite[Theorem 4.16]{kirch00}.

Now assume that $H$ is a saturated hereditary subset of $\go$. We consider the quotient ultragraph $\g/(H,\emptyset)$ and take an arbitrary $[A]\in \pgo\setminus \{[\emptyset]\}$. If there is no loops $\alpha\in r_\g^{-1}(\go\setminus H)$ with $A\geq s_\g(\alpha)$, then the ideal $I_{[A]}:=\langle q_{[A]}\rangle\unlhd C^*(\g/(H,\emptyset))$ is AF. Thus $q_{[A]}$ is not infinite and $C^*(\g)$ contains a non-properly infinite projection, contradicting \cite[Theorem 4.16]{kirch00}. Moreover, we notice that for any $w\in B_H$, $[w']$ is a sink in $\g/(H,\emptyset)$ and the projection $q_{[w']}$ is not infinite, which is impossible.

Conversely, suppose that $\g$ satisfies Condition (K) and the asserted properties hold for any saturated hereditary set $H$. To show that $C^*(\g)$ is purely infinite we apply \cite[Theorem 5.12]{kwa17} for the pure infiniteness of Fell bundles. Let $\mathcal{B}=\{B_n\}_{n\in \mathbb{Z}}$ be the natural Fell bundle in $C^*(\g)$. Corollary \ref{cor6.6} says that $\mathcal{B}$ is residually aperiodic. Moreover, every projection in $B_0$ is Murray-von Neumann equivalent to a finite sum $\sum_{i=1}^n r_i s_{\alpha_i} p_{B_i} s_{\beta_i}^*$ of mutually orthogonal projections such that $|\alpha_i|=|\beta_i|$ for $1\leq i\leq n$. Note that each projection $s_{\alpha_i} p_{B_i} s_{\beta_i}^*$ is Murray-von Neumann equivalent to $\left(s_{\alpha_i}p_{B_i}\right)^*\left(p_{B_i}s_{\beta_i}\right)$ which equals to zero unless $\alpha_i=\beta_i$. Hence, in view of \cite[Lemma 5.13]{kwa17}, it suffices to show that every nonzero projection of the form $s_\mu p_B s_\mu^*$ is properly infinite.

Let $I_{(H,\emptyset)}$ be an ideal in $C^*(\g)$ such that $s_\mu p_B s_\mu^*\notin I_{(H,\emptyset)}$. Then $B\cap r_\g(\mu)\in \go\setminus H$. Assume $C^*(\g/(H,\emptyset))=C^*(t_e,q_{[A]})$ and let $q:C^*(\g)\rightarrow C^*(\g/(H,\emptyset))$ be the canonical quotient map by Proposition \ref{prop5.1}. Then $q(s_\mu p_B s_\mu^*)=t_\mu q_{[B]}t_\mu^*\neq 0$. By hypothesis, there are a path $\lambda$ and a loop $\alpha\in r_\g^{-1}(\go\setminus H)$ such that $s_\g(\lambda)\in B\cap r_\g(\mu)$ and $s_\g(\alpha)\in r_\g(\lambda)$. Since $\g$ satisfies Condition (K), $\alpha$ has an exit $f$ in $r^{-1}(\go\setminus H)$. Thus we have
$$\left(t_\alpha q_{s(\alpha)}\right)\left(t_\alpha q_{s(\alpha)}\right)^* +t_f t_f^*\leq q_{s(\alpha)},$$
and since
$$\left(t_\alpha q_{s(\alpha)}\right)\left(t_\alpha q_{s(\alpha)}\right)^*\sim \left(t_\alpha q_{s(\alpha)}\right)^*\left(t_\alpha q_{s(\alpha)}\right)=q_{s(\alpha)},$$
it turns out that $q_{s(\alpha)}$ is an infinite projection in $C^*(\g/(H,\emptyset))\cong C^*(\g)/I_{(H,\emptyset)}$. On the other hand, the fact
$$\left(t_{\mu\lambda} q_{s(\alpha)}\right)^* t_\mu q_{[B]} t_\mu^* \left(t_{\mu\lambda} q_{s(\alpha)}\right)=q_{s(\alpha)}$$
says that $q_{s(\alpha)}\precsim t_\mu q_{[B]}t_\mu^*$ (see \cite[Proposition 2.4]{ror92}), and thus $t_\mu q_{[B]} t_\mu^*$ is infinite by \cite[Lemma 3.17]{kirch00}. It follows that  $s_\mu p_B s_\mu^*$ is a properly infinite projection. Now apply \cite[Theorem 5.11(ii)]{kwa17} to conclude that the $C^*$-algebra $C^*(\g)\cong C_r^*(\mathcal{B})$ is purely infinite.
\end{proof}

%%%%%%%%%%%%%%%%%%%%%%%%%%%%%%%%%%%%%%%%%%%%%%%%

\end{document}